\documentclass[12pt]{amsart}

\usepackage[margin=1in]{geometry}
\usepackage{amsmath,amsthm,amssymb}
\usepackage{dsfont}
\usepackage{color}
\usepackage{hyperref}

\hypersetup{
    colorlinks=true, 
    linkcolor=blue, 
    urlcolor=red, 
    linktoc=all 
}

\input xy
\xyoption{all}

\newenvironment{remark}{\noindent\textbf{Remark:}\hspace{1em}}{}

\newcommand{\C}{\mathbb{C}}
\newcommand{\PP}{\mathbb{P}}
\newcommand{\Z}{\mathbb{Z}}
\newcommand{\one}{\mathds{1}}
\newcommand{\G}{\mathcal{G}}
\newcommand{\map}{\textrm{Map}}

\theoremstyle{plain}
\newtheorem{thm}{Theorem}[section]
\newtheorem{prop}[thm]{Proposition}
\newtheorem{lemma}[thm]{Lemma}
\newtheorem{cor}[thm]{Corollary}

\begin{document}

\title[Homotopy types of gauge groups over 4-manifolds]{Homotopy types of gauge groups over non-simply-connected closed 4-manifolds}
\author{Tseleung So}
\address{Mathematical Sciences, University of Southampton, SO17 1BJ, UK}
\email{tls1g14@soton.ac.uk}
\thanks{}



\date{}

\begin{abstract}
Let $G$ be a simple, simply-connected, compact Lie group and let $M$ be an orientable, smooth, connected, closed 4-manifold. In this paper we calculate the homotopy type of the suspension of $M$ and the homotopy types of the gauge groups of principal $G$-bundles over~$M$ when $\pi_1(M)$ is (1) $\Z^{*m}$, (2) $\Z/p^r\Z$, or (3) $\Z^{*m}*(*^n_{j=1}\Z/p^{r_j}_j\Z)$, where $p$ and the $p_j$'s are odd primes.
\end{abstract}

\maketitle

\section{Introduction}
Let $G$ be a topological group and let $M$ be a topological space. Given a principal $G$-bundle $P$ over $M$, the associated gauge group $\G(P)$ is the topological group of $G$-equivariant automorphisms of $P$ which fix~$M$. Atiyah, Bott and Gottlieb \cite{AB83, gottlieb72} showed that its classifying space $B\G(P)$ is homotopy equivalent to the connected component $\map_P(M, BG)$ of the mapping space $\map(M, BG)$ which contains the map inducing $P$. When $G$ is a simple, simply-connected, compact Lie group and $M$ is an orientable, smooth, connected, closed 4-manifold, it can be shown that the set of isomorphism classes of principal $G$-bundles over $M$ is in one-to-one correspondence with the homotopy set $[M, BG]\cong\Z$. If a principal $G$-bundle corresponds to an integer $t$, then we denote its associated gauge group by $\G_t(M)$. In \cite{theriault10a} Theriault shows that when $M$ is a simply-connected spin 4-manifold, there is a homotopy equivalence
\begin{equation}\label{simply_conn_decomp}
\G_t(M)\simeq\G_t(S^4)\times\prod^n_{i=1}\Omega^2G,
\end{equation}
where $n$ is the rank of $H^2(M; \Z)$. For the non-spin case this homotopy equivalence still holds when localized away from 2. As a result, the study of $\G_t(M)$ can be reduced to that of $\G_t(S^4)$, which has been being investigated over the last twenty years. Kishimoto, Kono and Tsutaya gave bounds on the numbers of distinct homotopy types of $\G_t(S^4)$ for all $G$ using localization at odd primes \cite{KKT14}. Moreover, the homotopy types of $\G_t(S^4)$ are classified for many cases. Let~$(a, b)$ be the greatest common divisor of $a$ and $b$. Then \cite{HK06, KK18, KTT17, kono91, theriault10b, theriault15, theriault17}
\begin{itemize}
\setlength\itemsep{1em}
\item	when $G=SU(2)$, there is a homotopy equivalence $\G_t(S^4)\simeq\G_s(S^4)$ if and only if~$(12, t)=(12, s)$;
\item	when $G=SU(3)$, there is a homotopy equivalence $\G_t(S^4)\simeq\G_s(S^4)$ if and only if~$(24, t)=(24, s)$;
\item	when $G=SU(5)$, there is a $p$-homotopy equivalence $\G_t(S^4)\simeq\G_s(S^4)$ if and only if~$(120, t)=(120, s)$ for any prime $p$;
\item	when $G=Sp(2)$, there is a $p$-local homotopy equivalence $\G_t(S^4)\simeq\G_s(S^4)$ if and only if $(40, t)=(40, s)$ for any prime $p$;
\item	when $G=SU(n)$, there is a $p$-local homotopy equivalence $\G_t(S^4)\simeq\G_s(S^4)$ if and only if $(n(n^2-1), t)=(n(n^2-1), s)$ for any odd prime $p$ such that $n\leq(p-1)^2+1$;
\item	when $G=Sp(n)$, there is a $p$-local homotopy equivalence $\G_t(S^4)\simeq\G_s(S^4)$ if and only if $(4n(2n+1), t)=(4n(2n+1), s)$ for any odd prime $p$ such that $2n\leq(p-1)^2+1$;
\item	when $G=G_2$, there is a $p$-local homotopy equivalence $\G_t(S^4)\simeq\G_s(S^4)$ if and only if~$(84, t)=(84, s)$ for any odd prime $p$,
\end{itemize}
In addition, a few cases of~$\G_t(\C\PP^2)$ are worked out \cite{KT96, theriault12}:
\begin{itemize}
\setlength\itemsep{1em}
\item	when $G=SU(2)$, there is a homotopy equivalence $\G_t(\C\PP^2)\simeq\G_s(\C\PP^2)$ if and only if~$(6, t)=(6, s)$;
\item	when $G=SU(3)$, there is a $p$-local homotopy equivalence $\G_t(\C\PP^2)\simeq\G_s(\C\PP^2)$ if and only if $(12, t)=(12, s)$ for any prime $p$.
\end{itemize}

On the other hand, very little is known about $\G_t(M)$ when $M$ is non-simply-connected. The goal of this paper is to study the homotopy types of $\G_t(M)$ for certain non-simply-connected 4-manifolds. To achieve this we need a homotopy decomposition statement.

\begin{thm}\label{1_G(M) decomp}
Suppose that $G$ is a simple, simply-connected, compact Lie group and $Y$ is a CW-complex of dimension at most 3. Let $\phi:Y\to M$ be a map such that $\Sigma\phi$ has a left homotopy inverse. Then we have
\[
\G_t(M)\simeq\G_t(C_{\phi})\times\map^*(Y, G),
\]
where $C_{\phi}$ is the cofiber of $\phi$.
\end{thm}

Using Theorem~\ref{1_G(M) decomp} we calculate the homotopy types of $\G_t(M)$ when $\pi_1(M)$ is (1) $\Z^{*m}$, (2)~$\Z/p^r\Z$, or (3) $\Z^{*m}*(*^n_{j=1}\Z/p^{r_j}_j\Z)$, where $p$ and the $p_j$'s are odd primes.

\begin{thm}\label{1_G(M) calculation}
Suppose that $G$ is a simple, simply-connected, compact Lie group and $M$ is an orientable, smooth, connected, closed 4-manifold.
\begin{itemize}
\item
If $\pi_1(M)=\Z^{*m}$ or $\Z/p^r\Z$, then $\G_t(M)$ is homotopy equivalent to a product of $\G_t(S^4)$ or $\G_t(\C\PP^2)$ and ``loop spaces'' on $G$.

\item
If $\pi_1(M)=\Z^{*m}*(*^n_{j=1}\Z/p^{r_j}_j\Z)$, then $\G_t(M)\times\prod^{2d}\Omega^2G$ is homotopy equivalent to a product of $\G_t(S^4)$ or $\G_t(\C\PP^2)$ and ``loop spaces'' on $G$ for some number $d$.
\end{itemize}
\end{thm}
The term ``loop spaces'' refers both to iterated based loop spaces $\Omega G$, $\Omega^2G$ and $\Omega^3G$ and modular loop spaces $\Omega G\{p^r\}$ and $\Omega^2G\{p^r\}$, where $G\{p^r\}$ is the homotopy fiber of the~$p^r$-power map on $G$. Explicit decompositions are stated in Section 3.

Theorem~\ref{1_G(M) calculation} shows that the homotopy type of $\G_t(M)$ is related to that of $\G_t(S^4)$ or~$\G_t(\C\PP^2)$ in these three cases. Combining Theorem~\ref{1_G(M) calculation} and the known results in \cite{HK06, kono91, KTT17, theriault10b, theriault17}, we have the following classification.

\begin{cor}\label{1_cor G(P) case12}
If $M$ is an orientable, smooth, connected, closed 4-manifold with~\mbox{$\pi_1(M)=\Z^{*m}$} or $\Z/p^r\Z$, then the followings hold:
\begin{itemize}
\setlength\itemsep{1em}
\item	when $G=SU(2)$, there is a homotopy equivalence $\G_t(M)\simeq\G_s(M)$ if and only if~$(12, t)=(12, s)$ for $M$ spin, and $(6,t)=(6,s)$ for $M$ non-spin;
\item	when $G=SU(3)$, there is a homotopy equivalence $\G_t(M)\simeq\G_s(M)$ if and only if~$(24, t)=(24, s)$ for $M$ spin; there is a $p$-local homotopy equivalence $\G_t(M)\simeq\G_s(M)$ if and only if~\mbox{$(12,t)=(12,s)$} for any prime $p$ and $M$ non-spin;
\item	when $G=SU(n)$, there is a $p$-local homotopy equivalence $\G_t(M)\simeq\G_s(M)$ if and only if $(n(n^2-1), t)=(n(n^2-1), s)$ for any odd prime $p$ such that $n\leq(p-1)^2+1$;
\item	when $G=Sp(2)$, there is a $p$-local homotopy equivalence $\G_t(M)\simeq\G_s(M)$ if and only if $(40, t)=(40, s)$ for any odd prime~$p$;
\item	when $G=G_2$, there is a $p$-local homotopy equivalence $\G_t(M)\simeq\G_s(M)$ if and only if~$(84, t)=(84, s)$ for any odd prime $p$.
\end{itemize}
\end{cor}

There is an analogous statement to Corollary~\ref{1_cor G(P) case12}.

\begin{cor}
If $M$ is an orientable, smooth, connected, closed 4-manifold with $\pi_1(M)=\Z^{*m}*(*^n_{j=1}\Z/p^{r_j}_j\Z)$, then the integral and $p$-local homotopy equivalences in Corollary~\ref{1_cor G(P) case12} hold for $\G_t(M)\times\prod^{2d}\Omega^2G$ for some number $d$.
\end{cor}

The structure of this paper is as follows. In Section 2 we prove the homotopy decomposition Theorem~\ref{1_G(M) decomp} and develop some useful lemmas. In particular, Theorem~\ref{1_G(M) decomp} is used to revise homotopy equivalence~(\ref{simply_conn_decomp}) which is often referred to during the calculations in Section 3. In Section 3 we give the homotopy types of $\Sigma M$ and $\G_t(M)$ when $\pi_1(M)$ is either $\Z^{*m}$, $\Z/p^r\Z$ or $\Z^{*m}*(*^n_{j=1}\Z/p^{r_j}_j\Z)$, where $p$ and the $p_j$'s are odd primes.

The author would like to thank the reviewer for carefully reading the paper and giving many constructive comments.

\section{A homotopy decomposition of gauge groups}

\subsection{A homotopy decomposition and some useful lemmas}
We are going to extend the homotopy decomposition~(\ref{simply_conn_decomp}) to a more general situation. Suppose that $M$ is an orientable 4-dimensional CW-complex that is constructed by attaching a 4-cell onto a 3-dimensional CW-complex $M_3$ by an attaching map $f:S^3\to M_3$. Let $u:BG\to K(\Z, 4)$ be a generator of $H^4(BG)\cong\Z$. Since it is a 5-equivalence, $u_*:[M, BG]\to[M, K(\Z, 4)]=\Z$ is a bijection. If a principal $G$-bundle over $M$ corresponds to some integer $t$, then we denote its associated gauge group by $\G_t(M)$.

In \cite{theriault10a}, the homotopy decomposition~(\ref{simply_conn_decomp}) is obtained as a consequence of the attaching map $f$ of the 4-cell in $M$ having the property that $\Sigma f$ is null-homotopic. Observe that $\Sigma f$ is the connecting map of the cofibration sequence
\[
M_3\longrightarrow M\longrightarrow S^4,
\]
where $M_3$ is the 3-skeleton of $M$. From the point of view of homotopy theory, it can be replaced by a cofibration sequence
\begin{equation}\label{eqn_cofib phi}
Y\overset{\phi}{\longrightarrow}M\overset{q}{\longrightarrow}C_{\phi}
\end{equation}
for some space $Y$ and some map $\phi:Y\to M$, with a connecting map that is null-homotopic. Here~$q$ is the quotient map and $C_{\phi}$ is the cofiber of $\phi$. The nullity condition is equivalent to~$\Sigma\phi$ having a left homotopy inverse. If we further restrict the dimension of $Y$ to be at most~3, then by Cellular Approximation Theorem $\phi$ is homotopic to $\imath\circ\varphi$, where $\varphi:Y\to M_3$ is a map on $M_3$ and $\imath:M_3\to M$ is the inclusion. The existence of a left homotopy inverse of $\Sigma\phi$ imposes a strong condition on $\Sigma M$.

\begin{lemma}\label{2_Y criteria}
Let $Y$ be a CW-complex of dimension at most 3 and let $\phi:Y\to M$ be a map. If $\phi$ is homotopic to $\imath\circ\varphi$ for some map $\varphi:Y\to M_3$, then the following are equivalent:
\begin{enumerate}
\item
$\Sigma\phi:\Sigma Y\to \Sigma M$ has a left homotopy inverse $\psi$;

\item
$\Sigma\varphi$ has a left homotopy inverse $\psi'$ and $\psi'\circ\Sigma f$ is null-homotopic;

\item
$\Sigma M$ is homotopy equivalent to $\Sigma Y\vee\Sigma C_{\phi}$ and $\Sigma\phi$ is homotopic to the inclusion;

\item
$\Sigma M_3$ is homotopy equivalent to $\Sigma Y\vee\Sigma C_{\varphi}$ where $C_{\varphi}$ is the cofiber of $\varphi$, $\Sigma\varphi$ is homotopic to the inclusion and $\psi'\circ\Sigma f$ is null-homotopic.
\end{enumerate}
\end{lemma}

\begin{proof}
First we show that Condition (1) and (2) are equivalent. If $\Sigma\phi$ has a left homotopy inverse $\psi$, then $\psi'=\psi\circ\Sigma\imath$ is a left homotopy inverse of $\Sigma\varphi$ and $\psi'\circ\Sigma f=\psi\circ\Sigma\imath\circ\Sigma f$ is null-homotopic since $S^4\overset{\Sigma f}{\longrightarrow}\Sigma M_3\overset{\Sigma\imath}{\longrightarrow}\Sigma M$ is a cofibration sequence. Conversely, assume Condition (2). Consider the homotopy commutative diagram
\[\xymatrix{
S^4\ar[r]^-{\Sigma f}	&\Sigma M_3\ar[r]^-{\Sigma\imath}\ar[d]^-{\psi'}	&\Sigma M\ar@{-->}[dl]^-{\psi}\\
						&\Sigma Y											&
}\]
By hypothesis $\psi'\circ\Sigma f$ is null-homotopic, so $\psi'$ has an extension $\psi$. Then we have
\begin{eqnarray*}
\psi\circ\Sigma\phi
&\simeq&\psi\circ\Sigma(\imath\circ\varphi)\\
&\simeq&(\psi\circ\Sigma\imath)\circ\Sigma\varphi\\
&\simeq&\psi'\circ\Sigma\varphi\\
&\simeq&\one_{\Sigma Y},
\end{eqnarray*}
where $\one_{\Sigma Y}$ is the identity map on $\Sigma Y$. Therefore $\psi$ is a left homotopy inverse of $\Sigma\phi$.

Now we show that Conditions (1) and (3) are equivalent. If $\Sigma\phi$ has a left homotopy inverse~$\psi$, then let $h$ be the composition
\[
h:\Sigma M\overset{\sigma}{\longrightarrow}\Sigma M\vee\Sigma M\overset{\psi\vee\Sigma q}{\longrightarrow}\Sigma Y\vee\Sigma C_{\phi},
\]
where $\sigma$ is the comultiplication. Observe that $h$ induces an isomorphism $H^*(\Sigma Y\vee\Sigma C_{\phi})\to H^*(\Sigma M)$. Since these spaces are suspensions and are simply-connected, $h$ is a homotopy equivalence by Whitehead Theorem. Therefore $\Sigma M$ is homotopy equivalent to~\mbox{$\Sigma Y\vee\Sigma C_{\phi}$}. Moreover, since $\Sigma\phi$ is a co-H-map,
\begin{eqnarray*}
h\circ\Sigma\phi
&\simeq&(\psi\vee\Sigma q)\circ\sigma\circ\Sigma\phi\\
&\simeq&(\psi\vee\Sigma q)\circ(\Sigma\phi\vee\Sigma\phi)\circ\sigma\\
&\simeq&(\psi\circ\Sigma\phi\vee\Sigma q\circ\Sigma\phi)\circ\sigma\\
&\simeq&(\one_{\Sigma Y}\vee\ast)\circ\sigma
\end{eqnarray*}
is the inclusion $\Sigma Y\hookrightarrow\Sigma Y\vee\Sigma C_{\phi}$. Conversely, assume Condition (3). Let $h$ be a homotopy equivalence from $\Sigma M$ to $\Sigma Y\vee\Sigma C_{\phi}$ and let $\psi$ be the composition
\[
\psi:\Sigma M\overset{h}{\longrightarrow}\Sigma Y\vee\Sigma C_{\phi}\overset{p}{\longrightarrow}\Sigma Y.
\]
where $p$ is the pinch map. By hypothesis, $h\circ\Sigma\phi$ is homotopic to the inclusion~\mbox{$\Sigma Y\to\Sigma Y\vee\Sigma C_{\phi}$}, so we have
\[
\psi\circ\Sigma\phi
\simeq p\circ h\circ\Sigma\phi
\simeq\one_{\Sigma Y}
\]
and $\psi$ is a left homotopy inverse of $\Sigma\phi$.

The equivalence between (2) and (4) can be shown similarly.
\end{proof}

We can extend the cofibration~(\ref{eqn_cofib phi}) to the homotopy commutative diagram whose rows and columns are cofibration sequences
\begin{equation}\label{diagram_CW structure of C_phi}
\xymatrix{
\ast\ar[r]\ar[d]			&Y\ar@{=}[r]\ar[d]^-{\varphi}		&Y\ar[r]\ar[d]^-{\phi}		&\ast\ar[d]\\
S^3\ar[r]^-{f}\ar@{=}[d]	&M_3\ar[r]^-{\imath}\ar[d]^-{q'}	&M\ar[r]^-{p}\ar[d]^-{q}	&S^4\ar@{=}[d]\\
S^3\ar[r]^-{f'}				&C_{\varphi}\ar[r]					&C_{\phi}\ar[r]^-{p'}		&S^4
}
\end{equation}
where $q', p$ and $p'$ are the quotient maps and $f'=q'\circ f$. The bottom row implies that $C_{\phi}$ is constructed by attaching a 4-cell onto $C_{\varphi}$ via $f'$. The generator $u$ of $H^4(BG)$ induces a bijection between $[C_{\phi}, BG]$ and $[C_{\phi}, K(\Z,4)]\cong\Z$, so any principal $G$-bundle over $C_{\phi}$ corresponds to some integer $t$. Denote the associated gauge group by $\G_t(C_{\phi})$. We want to compare $\G_t(M)$ and $\G_t(C_{\phi})$ via the pullback $q^*$.

\begin{lemma}\label{2_map_t}
Assume the conditions in Lemma~\ref{2_Y criteria}. Then $q^*:[C_{\phi}, BG]\to[M,BG]$ is a group isomorphism.
\end{lemma}

\begin{proof}
The naturality of $u^*$ implies the commutative diagram
\[\xymatrix{
[C_{\phi}, BG]\ar[r]^-{q^*}\ar[d]^-{u^*}	&[M, BG]\ar[d]^-{u^*}\\
H^4(C_{\phi})\ar[r]^-{q^*}					&H^4(M).
}\]
Since the group structures on $[M, BG]$ and $[C_{\phi}, BG]$ are induced by bijections
\[
\begin{array}{c c c}
u^*:[M, BG]\to H^4(M)\cong\Z
&\text{and}
&u^*:[C_{\phi}, BG]\to H^4(C_{\phi})\cong\Z,
\end{array}
\]
it suffices to show that $q^*:H^4(C_{\phi})\to H^4(M)$ is an isomorphism.

The cofibration~(\ref{eqn_cofib phi}) induces an exact sequence
\[
H^4(\Sigma M)\overset{(\Sigma\phi)^*}{\longrightarrow}H^4(\Sigma Y)\overset{\epsilon^*}{\longrightarrow}H^4(C_{\phi})\overset{q^*}{\longrightarrow}H^4(M)\overset{\phi^*}{\longrightarrow}H^4(Y),
\]
where $\epsilon^*$ is induced by the connecting map $\epsilon:C_{\phi}\to\Sigma Y$. Since $\Sigma\phi$ has a left homotopy inverse, $\epsilon$ is null-homotopic and $\epsilon^*$ is trivial. Also, the dimension of $Y$ is at most 3, so $H^4(Y)$ is trivial. By exactness $q^*:H^4(C_{\phi})\to H^4(M)$ is an isomorphism.
\end{proof}

\begin{lemma}\label{lemma_fiber of q*}
For any integer $t$, let $\map^*_t(M, BG)$ and $\map^*_t(C_{\phi}, BG)$ be the connected components of $\map^*(M, BG)$ and $\map^*(C_{\phi}, BG)$ containing $t\in[M, BG]\cong[C_{\phi}, BG]$ respectively. Then
\[
\map^*(\Sigma Y, BG)\overset{\epsilon^*}{\longrightarrow}\map^*_t(C_{\phi}, BG)\overset{q^*}{\longrightarrow}\map^*_t(M, BG)
\]
is a homotopy fibration sequence.
\end{lemma}

\begin{proof}
First, $q^*:\map^*_t(C_{\phi}, BG)\to\map^*_t(M, BG)$ is well defined by Lemma~\ref{2_map_t}. Second, the cofibration~(\ref{eqn_cofib phi}) induces a homotopy fibration sequence
\[
\map^*(\Sigma Y, BG)\overset{\epsilon^*}{\longrightarrow}\map^*(M, BG)\overset{q^*}{\longrightarrow}\map^*(C_{\phi}, BG).
\]
Since $\epsilon^*$ is null-homotopic, $\map^*(\Sigma Y, BG)$ is mapped into $\map^*_0(M, BG)$ and hence is the homotopy fiber of $q^*:\map^*_0(M, BG)\to\map^*_0(C_{\phi}, BG)$. To show that it is true for any $t\in\Z$, let $\alpha_t:\map^*_0(M, BG)\to\map^*_t(M, BG)$ be a map sending a pointed map $g:M\to BG$ to the composition
\[
\alpha_t(g):M\overset{c}{\longrightarrow}M\vee S^4\overset{g\vee t}{\longrightarrow}BG\vee BG\overset{\triangledown}{\longrightarrow}BG,
\]
where $c$ is the coaction map of $M$ and $\triangledown$ is the folding map, and let $\beta_t:\map^*_0(C_{\phi}, BG)\to\map^*_t(C_{\phi}, BG)$ be a map defined similarly. Since $\alpha_t$ and $\beta_t$ are homotopy equivalences, it suffices to prove the commutate diagram
\[\xymatrix{
\map^*_0(C_{\phi}, BG)\ar[r]^-{q^*}\ar[d]^-{\alpha_t}	&\map^*_0(M, BG)\ar[d]^-{\beta_t}\\
\map^*_t(C_{\phi}, BG)\ar[r]^-{q^*}						&\map^*_t(M, BG).
}\]

Consider the homotopy commutative diagram
\[\xymatrix{
M\ar[rrr]^-{c}\ar[dr]^-{p}\ar[dd]_-{q}	&	&	&M\vee S^4\ar[dd]^-{q\vee\one}\ar[dl]_-{p\vee\one}\\
	&S^4\ar[r]^-{\sigma}	&S^4\vee S^4	&\\
C_{\phi}\ar[rrr]_-{c'}\ar[ur]_-{p'}		&	&	&C_{\phi}\vee S^4\ar[ul]^-{p'\vee\one}
}\]
where $c'$ is the coaction map, and $p$ and $p'$ are the quotient maps, and $\sigma$ is the comultiplication of $S^4$. The left and the right triangles are due to the bottom right square in diagram~(\ref{diagram_CW structure of C_phi}), the top and the bottom quadrangles are due to the property of coaction maps. Extend it to get the cofibration diagram
\[\xymatrix{
M\ar[r]^-{c}\ar[d]^-{q}	&M\vee S^4\ar[r]^-{q^*g\vee t}\ar[d]^-{q\vee\one}	&BG\vee BG\ar[r]^-{\triangledown}\ar@{=}[d]	&BG\ar@{=}[d]\\
C_{\phi}\ar[r]^-{c'}	&C_{\phi}\vee S^4\ar[r]^-{g\vee t}			&BG\vee BG\ar[r]^-{\triangledown}	&BG
}\]
The upper row around the diagram is $\beta_t(q^*g)$, while the lower row around the diagram is~$q^*\alpha_t(g)$. Therefore $q^*$ commutes with $\alpha_t$ and $\beta_t$ and the asserted statement follows.
\end{proof}

\begin{thm}\label{gaugedecomp}
Let $Y$ be a CW-complex of dimension at most 3 and let $\phi:Y\to M$ be a map. If $\phi$ satisfies one of the four conditions in Lemma~\ref{2_Y criteria}, then there are homotopy equivalences
\[
\begin{array}{c c c}
\Sigma M\simeq\Sigma C_{\phi}\vee\Sigma Y
&\text{and}
&\G_t(M)\simeq\G_t(C_{\phi})\times\map^*(Y, G).
\end{array}
\]
\end{thm}

\begin{proof}
Let $\map_t(M, BG)$ and $\map_t(C_{\phi}, BG)$ be the connected components of $\map(M, BG)$ and $\map(C_{\phi}, BG)$ containing $t\in[M, BG]\cong[C_{\phi}, BG]$. There are evaluation fibrations
\[
\begin{array}{c c c}
\map^*_t(M, BG)\to\map_t(M, BG)\to BG
&\text{and}
&\map^*_t(C_{\phi}, BG)\to\map_t(C_{\phi}, BG)\to BG.
\end{array}
\]
By Lemma~\ref{2_map_t} $q^*:[C_{\phi}, BG]\to[M, BG]$ is an isomorphism, so there is a homotopy commutative diagram whose rows are fibration sequences
\begin{equation}\label{1-conn naturality}
\xymatrix{
G\ar[r]\ar@{=}[d]	&\map^*_t(C_{\phi}, BG)\ar[r]\ar[d]^-{q^*}	&\map_t(C_{\phi}, BG)\ar[r]\ar[d]^-{q^*}	&BG\ar@{=}[d]\\
G\ar[r]				&\map^*_t(M, BG)\ar[r]					&\map_t(M, BG)\ar[r]				&BG
}
\end{equation}
As in \cite{theriault10a}, from the left square in~(\ref{1-conn naturality}) we obtain a homotopy fibration diagram
\begin{equation}\label{1_gauge gp CD}
\xymatrix{
\ast\ar[r]\ar[d]						&\Omega\map^*_t(M, BG)\ar@{=}[r]\ar[d]		&\Omega\map^*_t(M, BG)\ar[d]\\
\G_t(C_{\phi})\ar[r]^-{q^*}\ar@{=}[d]	&\G_t(M)\ar[r]^-{h}\ar[d]					&\map^*(\Sigma Y, BG)\ar[d]^-{\epsilon^*}\\
\G_t(C_{\phi})\ar[r]\ar[d]				&G\ar[r]\ar[d]								&\map^*_t(C_{\phi}, BG)\ar[d]\\
\ast\ar[r]								&\map^*_t(M, BG)\ar@{=}[r]					&\map^*_t(M, BG)
}
\end{equation}
The right column is due to Lemma~\ref{lemma_fiber of q*}. The nullity of $\epsilon^*$ implies that $h$ has a right homotopy inverse. The group multiplication in $\G_t(M)$ then gives a homotopy equivalence
\[
\Phi:\G_t(C_{\phi})\times\map^*(\Sigma Y, BG)\longrightarrow\G_t(M)\times\G_t(M)\longrightarrow\G_t(M),
\]
so $\G_t(M)$ is homotopy equivalent to $\G_t(C_{\phi})\times\map^*(Y, G)$.
\end{proof}

When calculating $\G_t(M)$, we will use Theorem~\ref{gaugedecomp} several times and apply it to $M$ and~$C_{\phi}$ in some cases. So here we establish some facts about $C_{\phi}$.

\begin{lemma}\label{lemma_f' and f}
Let $q':M\to C_{\phi}$ be the quotient map and let $f'$ be the composition
\[
S^3\overset{f}{\longrightarrow}M_3\overset{q'}{\longrightarrow}C_{\varphi}.
\]
Then $\Sigma f'$ is null-homotopic if and only if $\Sigma f$ is null-homotopic.
\end{lemma}

\begin{proof}
The necessity is obvious. Assume that $\Sigma f'$ is null-homotopic. Since the composition
\[
h:\Sigma M_3\overset{\sigma}{\longrightarrow}\Sigma M_3\vee\Sigma M_3\overset{\psi\vee\Sigma q'}{\longrightarrow}\Sigma Y\vee\Sigma C_{\varphi}
\]
is a homotopy equivalence, it suffices to show that $h\circ\Sigma f$ is null-homotopic. Consider the homotopy commutative diagram
\[\xymatrix{
S^4\ar[r]^-{\sigma}\ar[d]^-{\Sigma f}	&S^4\vee S^4\ar[d]^-{\Sigma f\vee\Sigma f}				&\\
\Sigma M_3\ar[r]^-{\sigma}				&\Sigma M_3\vee\Sigma M_3\ar[r]^-{\psi\vee\Sigma q'}	&\Sigma Y\vee\Sigma C_{\varphi}
}\]
where the two $\sigma$'s are the comultiplications of $S^4$ and $\Sigma M_3$. The lower direction around the diagram is $h\circ\Sigma f$, and the upper direction around the diagram is
\[
(\psi\circ\Sigma f\vee\Sigma q'\circ\Sigma f)\circ\sigma\simeq(\psi\circ\Sigma f\vee\Sigma f')\circ\sigma
\]
By hypothesis $\Sigma f'$ is null-homotopic, and by Lemma~\ref{2_Y criteria}(4) $\psi\circ\Sigma f$ is null-homotopic. Therefore $h\circ\Sigma f$ is null-homotopic and hence so is $\Sigma f$.
\end{proof}

\begin{lemma}\label{q isom on H free}
Let $H^2_{\text{free}}(X)$ be the free part of $H^2(X)$ for any space $X$. If $H^2(Y)$ is torsion, then $q^*:H^2_{\text{free}}(C_{\phi})\to H^2_{\text{free}}(M)$ is an isomorphism. Moreover, if for any $\alpha'\in H^2_{\text{free}}(M)$ there exists $\beta'\in H^2_{\text{free}}(M)$ such that $\alpha'\cup\beta'\in H^4(M)$ is a generator, then for any $\alpha\in H^2_{\text{free}}(C_{\phi})$ there exists $\beta\in H^2_{\text{free}}(C_{\phi})$ such that $\alpha\cup\beta\in H^4(C_{\phi})$ is a generator.
\end{lemma}

\begin{proof}
Cofibration~(\ref{eqn_cofib phi}) induces the long exact sequence of cohomology groups
\[
\cdots\longrightarrow H^k(C_{\phi})\overset{(\Sigma q)^*}{\longrightarrow}H^k(M)\overset{(\Sigma\phi)^*}{\longrightarrow}H^k(Y)\longrightarrow\cdots
\]
Since $\Sigma\phi$ has a left homotopy inverse, the sequence splits for $k\geq1$ and we have
\[
H^k(M)\cong H^k(Y)\oplus H^k(C_{\phi}).
\]
By hypothesis $H^2(Y)$ is torsion, so $H^2_{\text{free}}(C_{\phi})$ is isomorphic to $H^2_{\text{free}}(M)$.

For any $\alpha\in H^2_{\text{free}}(C_{\phi})$, $q^*(\alpha)$ is in $H^2_{\text{free}}(M)$. By hypothesis there exists $\beta'\in H^2_{\text{free}}(M)$ such that $q^*(\alpha)\cup\beta'\in H^4(M)$ is a generator. Since $q^*:H^2_{\text{free}}(C_{\phi})\to H^2_{\text{free}}(M)$ is an isomorphism, there exists $\beta\in H^2_{\text{free}}(C_{\phi})$ such that $q^*(\beta)=\beta'$. Therefore we have $q^*(\alpha)\cup\beta'=q^*(\alpha\cup\beta)$. Observe that $q^*:H^4(C_{\phi})\to H^4(M)$ is an isomorphism since $Y$ has dimension at most 3. It follows that $\alpha\cup\beta$ is a generator of $H^4(C_{\phi})$.
\end{proof}

The second part of Lemma~\ref{q isom on H free} says that the cup product on $H^2_{\text{free}}(C_{\phi})$ is unimodular if the cup product on $H^2_{\text{free}(M)}$ is unimodular, which follows from Poincar\'{e} Duality when~$M$ is an orientable compact manifold. Furthermore, if $C_{\phi}$ has a subcomplex $Y'$ satisfying Theorem~\ref{gaugedecomp}, then the cup product on $H^2_{\text{free}}(C_{\phi}/Y')$ is still unimodular.

Next we consider two variations of Theorem~\ref{gaugedecomp} when $M$ has a special structure.

\begin{lemma}\label{2_Y1 vee Y2}
Suppose that $M_3$ is homotopy equivalent to $Z\vee Z'$. Let $Y$ and $Y'$ be CW-complexes of dimension at most 3, and let $\varphi:Y\to Z$ and $\varphi':Y'\to Z'$ be maps. If $\Sigma\varphi$ and~$\Sigma\varphi'$ have left homotopy inverses $\psi$ and $\psi'$ respectively and the compositions
\[
\begin{array}{c c c}
S^4\overset{\Sigma f}{\longrightarrow}\Sigma M_3\overset{pinch}{\longrightarrow}\Sigma Z\overset{\psi}{\longrightarrow}\Sigma Y
&\text{and}
&S^4\overset{\Sigma f}{\longrightarrow}\Sigma M_3\overset{pinch}{\longrightarrow}\Sigma Z'\overset{\psi'}{\longrightarrow}\Sigma Y',
\end{array}
\]
are null-homotopic, then we have
\[
\begin{array}{c c c}
\Sigma M\simeq\Sigma M'\vee\Sigma Y\vee\Sigma Y'
&\text{and}
&\G_t(M)\simeq\G_t(M')\times\map^*(\Sigma Y, BG)\times\map^*(\Sigma Y', BG)
\end{array}
\]
where $M'$ is the cofiber of the map $Y\vee Y'\overset{\varphi\vee\varphi'}{\longrightarrow}Z\vee Z'\hookrightarrow M$.
\end{lemma}

\begin{proof}
Let $\Phi$ be the composition
\[
\Phi:Y\vee Y'\overset{\varphi\vee\varphi'}{\longrightarrow}Z\vee Z'\simeq M_3.
\]
The map $\psi\vee\psi':\Sigma Z\vee\Sigma Z'\to\Sigma Y\vee\Sigma Y'$ is a left homotopy inverse of $\Sigma\Phi$. We show that~$(\psi\vee\psi')\circ\Sigma f$ is null-homotopic, implying that $\Phi$ satisfies the hypothesis of Theorem~\ref{gaugedecomp}.

Notice that the composition
\[
h:\Sigma M_3\overset{\sigma}{\longrightarrow}\Sigma M_3\vee\Sigma M_3\overset{p_1\vee p_2}{\longrightarrow}\Sigma Z\vee\Sigma Z'
\]
is a homotopy equivalence, where $p_1:\Sigma M_3\to\Sigma Z$ and $p_2:\Sigma M_3\to\Sigma Z'$ are the pinch maps. Since $\Sigma f$ is a co-H-map, we have
\begin{eqnarray*}
(\psi\vee\psi')\circ h\circ\Sigma f
&\simeq&(\psi\vee\psi')\circ (p_1\vee p_2)\circ\sigma\circ\Sigma f\\
&\simeq&(\psi\vee\psi')\circ (p_1\vee p_2)\circ(\Sigma f\vee\Sigma f)\circ\sigma\\
&\simeq&(\psi\circ p_1\circ\Sigma f\vee\psi'\circ p_2\circ\Sigma f)\circ\sigma
\end{eqnarray*}
which is null-homotopic by assumption. Therefore $(\psi\vee\psi')\circ\Sigma f$ is null-homotopic and Theorem~\ref{gaugedecomp} implies the asserted statement.
\end{proof}

\begin{lemma}\label{connected sum}
Suppose $M\cong X\# X'$ where $X$ and $X'$ are orientable, smooth, connected, closed 4-manifolds. Let $Y$ and $Y'$ be CW-complexes of dimensions at most 3 and let $\varphi:Y\to X_3$ and $\varphi':Y'\to X'_3$ be maps satisfying the hypothesis of Theorem~\ref{gaugedecomp}. Then we have
\[
\begin{array}{c c c}
\Sigma M\simeq\Sigma M'\vee\Sigma Y\vee\Sigma Y'
&\text{and}
&\G_t(M)\simeq\G_t(M')\times\map^*(\Sigma Y, BG)\times\map^*(\Sigma Y', BG)
\end{array}
\]
where $M'$ is the cofiber of the inclusion $Y\vee Y'\overset{\varphi\vee\varphi'}{\longrightarrow}X_3\vee X'_3\hookrightarrow M$.
\end{lemma}

\begin{proof}
Let $f:S^3\to X_3$ and $f':S^3\to X'_3$ be the attaching maps of the 4-cells in $X$ and $X'$ respectively. By Lemma~\ref{2_Y criteria}, $\Sigma\varphi$ and $\Sigma\varphi'$ have left homotopy inverses $\psi$ and $\psi'$ and $\psi\circ\Sigma f$ and $\psi'\circ\Sigma f'$ are null-homotopic. We show that $\varphi$ and $\varphi'$ satisfy the hypothesis in Lemma~\ref{2_Y1 vee Y2}.

The 3-skeleton of $M$ is $X_3\vee X'_3$ and the attaching map of the 4-cell is
\[
f_{\#}:S^3\overset{\sigma}{\longrightarrow}S^3\vee S^3\overset{f\vee f'}{\longrightarrow}X_3\vee X'_3
\]
Let $p_1:\Sigma X_3\vee\Sigma X'_3\to\Sigma X_3$ and $p_2:\Sigma X_3\vee\Sigma X'_3\to\Sigma X'_3$ be the pinch maps. Consider the homotopy commutative diagram
\[\xymatrix{
S^4\ar@{=}[d]\ar[r]^-{\Sigma f_{\#}}	&\Sigma X_3\vee\Sigma X'_3\ar[d]^-{p_1}	&\\
S^4\ar[r]^-{\Sigma f}					&\Sigma X_3\ar[r]^-{\psi}				&\Sigma Y
}\]
Since $\psi\circ\Sigma f$ is null-homotopic, the composition $\psi\circ p_1\circ\Sigma f_{\#}$ is null-homotopic. Similarly~\mbox{$\psi'\circ p_2\circ\Sigma f_{\#}$} is null-homotopic. Then Lemma~\ref{2_Y1 vee Y2} implies the statement.
\end{proof}

\subsection{Gauge groups over simply-connected 4-manifolds}
Now we revise homotopy equivalence~(\ref{simply_conn_decomp}) using Theorem~\ref{gaugedecomp}. When $M$ is simply-connected, its 3-skeleton $M_3$ is homotopy equivalent to $\bigvee^n_{i=1}S^2$. If $\Sigma f$ is null-homotopic, then we can apply Theorem~\ref{gaugedecomp} by taking $Y$ to be the whole of $M_3$ and $\varphi:Y\to M_3$ to be the identity map and get the homotopy equivalence~(\ref{simply_conn_decomp}). In the following we assume the homotopy class of $\Sigma f$ is not trivial.

To distinguish the 2-spheres, denote the $i^{\text{th}}$ copy of $S^2$ in $M_3$ by $S^2_i$. Let $\eta:S^3\to S^2$ be the Hopf map and let $\eta_i$ be the composition
\[
\eta_i:S^3\overset{\eta}{\longrightarrow}S^2_i\hookrightarrow\bigvee^n_{i=1}S^2_i.
\]
We also denote the suspensions $\Sigma\eta$ and $\Sigma\eta_i$ by $\bar{\eta}$ and $\bar{\eta}_i$ for short.

\begin{lemma}\label{suspended f as sum}
If $M_3$ is homotopy equivalent to $\bigvee^n_{i=1}S^2_i$, then $\Sigma f$ is homotopic to~\mbox{$\sum^n_{i=1}a_i\bar{\eta}_i$} where $a_i\in\Z/2\Z$.
\end{lemma}

\begin{proof}
The Hilton-Milnor Theorem implies
\[
\pi_3(\bigvee^n_{i=n}S^2_i)\cong\bigoplus^n_{i=1}\pi_3(S^2_i)\oplus\bigoplus_{i\neq j}\pi_3(\Sigma S^1_i\wedge S^1_j),
\]
where $\pi_3(S^2_i)$ is generated by $\eta_i$ and $\pi_3(\Sigma S^1_i\wedge S^1_j)$ is generated by Whitehead products of the identity maps on $S^2_i$ and $S^2_j$. As an element in $\pi_3(\bigvee^n_{i=1}S^2_i)$, $f$ is homotopic to $\sum^n_{i=1}\tilde{a}_i\eta_i+w$ where $\tilde{a}_i$ is an integer and $w$ is a sum of Whitehead products in $\bigoplus_{i\neq j}\pi_3(S^2_i\wedge S^1_j)$. After suspension $\Sigma f$ is homotopic to $\sum^n_{i=1}a_i\bar{\eta}_i$, where $a_i\equiv\tilde{a}_i\pmod{2}$, since suspensions of Whitehead products are null-homotopic and $\bar{\eta}_i$ has order 2.
\end{proof}

If the homotopy class of $\Sigma f$ is not trivial, then at least one of the $a_i$'s is not zero. Relabelling the spheres if necessary, we may assume that $\Sigma f$ is homotopic to $\sum^m_{i=1}\bar{\eta}_i$ for some integer $m$ such that $1\leq m\leq n$. Then we can simplify this expression with the following lemma.

\begin{lemma}\label{2_simplify Sigma f}
If $M_3$ is homotopy equivalent to $\bigvee^n_{i=1}S^2_i$ and $\Sigma f$ is homotopic to $\sum^m_{i=1}\bar{\eta}_i$, then there is a map $\tilde{f}:S^3\to M_3$ such that its cofiber $C_{\tilde{f}}$ is homotopy equivalent to $M$ and~$\Sigma\tilde{f}$ is homotopic to $\bar{\eta}_1$. Moreover, $p_1\circ\tilde{f}$ is homotopic to $p_1\circ f$ where~\mbox{$p_1:\bigvee^n_{i=1}S^2_i\to S^2_1$} is the pinch map.
\end{lemma}

\begin{proof}
For each $1<j\leq m$, define a map $\xi_j:M_3\rightarrow M_3$ as follows. On $S^2_1\vee S^2_j$, $\xi_j$ is the composition
\[
S^2_1\vee S^2_j\overset{\sigma\vee\one}{\longrightarrow}S^2_{1}\vee S^2_{1}\vee S^2_j\overset{\one\vee\triangledown}{\longrightarrow}S^2_{1}\vee S^2_{j},
\]
where $\sigma$ is a comultiplication of $S^2_1$ and $\triangledown$ is the folding map of $S^2_{1}$ and $S^2_j$. On the remaining spheres, $\xi_j$ is the identity. Consider the homotopy cofibration diagram
\[\xymatrix{
S^3\ar@{=}[d]\ar[r]^-{f}	&M_3\ar[r]\ar[d]^-{\xi_j}	&M\ar[d]^-{\tilde{\xi}_j}\\
S^3\ar[r]^-{\xi_j\circ f}	&M_3\ar[r]					&C_{\xi_j\circ f}
}\]
where $\tilde{\xi}_j$ is an induced map and $C_{\xi_j\circ f}$ is the cofiber of $\xi_j\circ f$. Since $\xi_j$ is a homology isomorphism, so is $\tilde{\xi}_j$ by the 5-lemma. Therefore $\tilde{\xi}_j$ is a homotopy equivalence and $M$ is homotopy equivalent to $C_{\xi_j\circ f}$.

By Lemma~\ref{suspended f as sum}, $\Sigma(\xi_j\circ f)$ is homotopic to $\sum^n_{i=1}a_i\bar{\eta}_i$, where $a_i\in\Z/2\Z$. For $1\leq i\leq n$, let~\mbox{$p_i:\bigvee^n_{i=1}S^3_i\to S^3_i$} be the pinch map. Then $a_i$ is $p_i\circ\Sigma(\xi_j\circ f)$. By the definition of $\xi_j$, $a_i=1$ for $i\neq j$ and~\mbox{$1\leq i\leq m$}, and $a_i=0$ otherwise, that is
\[
\Sigma(\xi_j\circ f)
\simeq\bar{\eta}_1+\cdots+\bar{\eta}_{j-1}+\bar{\eta}_{j+1}+\cdots+\bar{\eta}_m
\]
Let $\tilde{f}$ be $\xi_m\circ\cdots\circ\xi_2\circ f$ and let $C_{\tilde{f}}$ be its cofiber. Then $\Sigma\tilde{f}$ is homotopic to $\bar{\eta}_1$ and $C_{\tilde{f}}$ is homotopy equivalent to $M$.

Lastly, observe that $p_1\circ\xi_j$ is homotopic to $p_1$. It follows that $p_1\circ\xi_j\circ f$ is homotopic to~\mbox{$p_1\circ f$} and so is $p_1\circ\tilde{f}$.
\end{proof}

\begin{lemma}\label{gauge_1-conn}
Let $M$ be a 4-dimensional simply-connected CW-complex. If $\Sigma f$ is nontrivial, then there are homotopy equivalences
\[
\begin{array}{c c c}
\displaystyle{\Sigma M\simeq\Sigma\C\PP^2\vee(\bigvee^{n-1}_{i=1}S^3)}
&\text{and}
&\displaystyle{\G_t(M)\simeq\G_t(C_{a\eta})\times\prod^{n-1}_{i=1}\Omega^2G},
\end{array}
\]
where $a$ is an odd integer, $n$ is the rank of $H^2(M)$ and $C_{a\eta}$ is the cofiber of $a\eta$.
\end{lemma}

\begin{proof}
By Lemma~\ref{suspended f as sum} $\Sigma f$ is homotopic to $\sum^m_{i=1}\bar{\eta}_i$ for some $m$ such that $1\leq m\leq n$. By Lemma~\ref{2_simplify Sigma f} there exists a map $\tilde{f}$ such that $\Sigma\tilde{f}$ is homotopic to $\bar{\eta}_1$ and its cofiber $\tilde{M}$ is homotopy equivalent to $M$. Replacing $f$ by $\tilde{f}$ and $M$ by $\tilde{M}$, we can assume $\Sigma f\simeq\bar{\eta}_i$. Use Theorem~\ref{gaugedecomp} by taking $Y$ to be $\bigvee^n_{i=2}S^2_i$ and $\phi:\bigvee^n_{i=2}S^2_i\to M$ to be the inclusion and get
\[
\begin{array}{c c c}
\displaystyle{\Sigma M\simeq\Sigma C_{\phi}\vee\left(\bigvee^{n-1}_{i=1}S^3\right)}
&\text{and}
&\displaystyle{\G_t(M)\simeq\G_t(C_{\phi})\times\prod^{n-1}_{i=1}\Omega^2G,}
\end{array}
\]
where $C_{\phi}$ is the cofiber of $\phi$. We need to show that $C_{\phi}$ is homotopy equivalent to $C_{a\eta}$ for some odd integer $a$ and $\Sigma C_{\phi}$ is homotopy equivalent to $\Sigma\C\PP^2$.

Consider the cofibration diagram
\[\xymatrix{
\ast\ar[r]\ar[d]								&S^3\ar@{=}[r]\ar[d]^-{f}					&S^3\ar[d]^-{p_1\circ f}\\
\bigvee^n_{i=2}S^2_i\ar[r]^-{\imath}\ar@{=}[d]	&\bigvee^n_{i=1}S^2_i\ar[d]\ar[r]^-{p_1}	&S^2_1\ar[d]\\
\bigvee^n_{i=2}S^2_i\ar[r]^-{\phi}				&M\ar[r]									&C_{\phi}
}\]
where $\imath:\bigvee^n_{i=2}S^2_i\to\bigvee^n_{i=1}S^2_i$ is the inclusion, $p_1:\bigvee^n_{i=2}S^3_i\to S^2_1$ is the pinch map. Since~$p_1\circ f$ is in $\pi_3(S^2)\cong\Z$, it is homotopic to $a\eta$ for some integer $a$. The right column implies that~$C_{\phi}$ is homotopy equivalent to the cofiber $C_{a\eta}$ of $a\eta$. Moreover, $\Sigma f$ is homotopic to $\bar{\eta}_1$, so $\Sigma(p_1\circ f)$ is homotopic to $\Sigma\eta$ and $a$ is an odd number. It follows that $\Sigma C_{\phi}$ is homotopy equivalent to the cofiber of $\Sigma\eta$, which is $\Sigma\C\PP^2$.
\end{proof}

We can modify the result of Lemma~\ref{gauge_1-conn} a bit better.

\begin{lemma}\label{CP_a}
Let $a$ be an odd number. Then we have
\[
\G_t(C_{a\eta})\times\Omega^2G\simeq\G_t(\C\PP^2)\times\Omega^2G.
\]
\end{lemma}

\begin{proof}
Let $g$ be the composition
\[
g:S^3\overset{\sigma}{\longrightarrow}S^3\vee S^3\overset{\eta\vee a\eta}{\longrightarrow}S^2\vee S^2
\]
and let $C_g$ be its cofiber. To distinguish the 2-spheres in the range, denote the $i^{\text{th}}$ copy by~$S^2_i$. Now we calculate $\G_t(C_g)$. By Lemma~\ref{2_simplify Sigma f}, we can assume that $\Sigma g$ is homotopic to $\bar{\eta}_1$ and~$p_1\circ g$ is homotopic to $\eta$. Use Theorem~\ref{gaugedecomp} by taking $Y$ to be $S^2_2$ and $\varphi:S^2_2\to S^2_1\vee S^2_2$ to be the inclusion and obtain $\G_t(C_g)\simeq\G_t(C_g/S^2_2)\times\Omega^2G$. Consider the homotopy cofibration diagram
\[\xymatrix{
\ast\ar[d]\ar[r]			&S^2_2\ar[d]^-{j}\ar@{=}[r]			&S^2_2\ar[d]^-{\varphi}\\
S^3\ar@{=}[d]\ar[r]^-{g}	&S^2_1\vee S^2_2\ar[r]\ar[d]^-{p_1}	&C_{g}\ar[d]\\
S^3\ar[r]^-{\eta}			&S^2_1\ar[r]						&\C\PP^2,
}\]
where $j$ is the inclusion and $p_1$ is the pinch map. The right column implies that $C_g/S^2_2$ is homotopy equivalent to $\C\PP^2$, so we have
\[
\G_t(C_g)\simeq\G_t(\C\PP^2)\times\Omega^2G.
\]

Similarly, let $g'$ be composition
\[
g':S^3\overset{\sigma}{\longrightarrow}S^3\vee S^3\overset{a\eta\vee\eta}{\longrightarrow}S^2_1\vee S^2_2
\]
and let $C_{g'}$ be its cofiber. By Lemma~\ref{2_simplify Sigma f}, $\Sigma g'$ is $\bar{\eta}_1$ and $p_1\circ g'$ is homotopic to $a\eta$ and~$p_1\circ g'$ is homotopic to $a\eta$. Use Theorem~\ref{gauge decomp_free} by taking $Y$ to be $S^2_2$ and $\varphi:S^2_2\to S^2_1\vee S^2_2$ and obtain~\mbox{$\G_t(M)\simeq\G_t(C_{g'}/S^2_2)\times\Omega^2G$}. Since $C_{g'}/S^2_2$ is homotopy equivalent to $C_{a\eta}$, we have
\[
\G_t(C_{g'})\simeq\G_t(C_{a\eta})\times\Omega^2G.
\]

Observe that $g=T\circ g'$, where $T:S^2_1\vee S^2_2\to S^2_2\vee S^2_1$ is the swapping map. Since $T$ is a homotopy equivalence, $C_{g'}$ and $C_g$ are homotopy equivalent. Combining the two homotopy equivalences gives the asserted lemma.
\end{proof}

\begin{prop}\label{2_eg_pi1=0}
Suppose that $M$ is a 4-dimensional simply-connected Poincar\'{e}-complex. Let the rank of $H^2(M)$ be $n$. If $\Sigma f$ is null-homotopic, then there are homotopy equivalences
\[
\begin{array}{c c c}
\displaystyle{\Sigma M\simeq S^5\vee(\bigvee^n_{i=1}S^3)}
&\text{and}
&\displaystyle{\G_t(M)\simeq\G_t(S^4)\times\prod^n_{i=1}\Omega^2G}.
\end{array}
\]
If $\Sigma f$ is nontrivial, then there are homotopy equivalences
\[
\begin{array}{c c c}
\displaystyle{\Sigma M\simeq\Sigma\C\PP^2\vee(\bigvee^{n-1}_{i=1}S^3)}
&\text{and}
&\displaystyle{\G_t(M)\simeq\G_t(\C\PP^2)\times\prod^{n-1}_{i=1}\Omega^2G}.
\end{array}
\]
\end{prop}

\begin{proof}
If $\Sigma f$ is null-homotopic, then use Theorem~\ref{gaugedecomp} by taking $Y$ to be $M_3$ and $\varphi:Y\to M_3$ to be the identity map and get the first two homotopy equivalences.

If $\Sigma f$ is nontrivial, by Lemma~\ref{gauge_1-conn} there are homotopy equivalences
\[
\begin{array}{c c c}
\displaystyle{\Sigma M\simeq\Sigma\C\PP^2\vee(\bigvee^{n-1}_{i=1}S^3)}
&\text{and}
&\displaystyle{\G_t(M)\simeq\G_t(C_{a\eta})\times\prod^{n-1}_{i=1}\Omega^2G},
\end{array}
\]
where $a$ is an odd integer. It suffices to show that we can replace $\G_t(C_{a\eta})$ by $\G_t(\C\PP^2)$.

When $n=1$, the Poincar\'{e} complex condition implies that $f$ has Hopf invariant equal to~1 or~\mbox{-1}. Therefore~$M$ is homotopy equivalent to $\C\PP^2$ and the statement holds. When~\mbox{$n\geq2$}, there is at least one copy of $\Omega^2G$ on the right hand side. By Lemma~\ref{CP_a}, we can replace~$\G_t(C_{a\eta})$ by~$\G_t(\C\PP^2)$ to obtain the proposition.
\end{proof}

It would be interesting to know if Lemma~\ref{CP_a} can be improved to $\G_t(C_{a\eta})\simeq\G_t(\C\PP^2)$. Observe that Proposition~\ref{2_eg_pi1=0} is an improvement on the homotopy equivalence~(\ref{simply_conn_decomp}) for the non-spin case in \cite{theriault10a}, which gives a decomposition only after localization away from 2.

\section{Gauge groups over non-simply-connected 4-manifolds}

From now on we assume that $M$ is an orientable, smooth, connected, closed 4 manifold. By Morse Theory, $M$ admits a CW-structure with one 4-cell \cite[Theorem 3.35]{matsumoto}. In this section we calculate the homotopy types of $\Sigma M$ and $\G_t(M)$ when $\pi_1(M)$ is (1) a free group~$\Z^{*m}$, (2) a cyclic group $\Z/p^r\Z$, or (3) a free product of types~\mbox{$(\Z^{*m})*(*^n_{j=1}\Z/p_j^{r_j}\Z)$}, where $p$ and the~$p_j$'s are odd primes. Our strategy is to apply Theorem~\ref{gaugedecomp} and its variations to decompose $\G_t(M)$ into a product of a gauge group of a simply-connected space, whose homotopy type is worked out in Proposition~\ref{2_eg_pi1=0}, and some complementary factors that do not depend on~$t$.

\subsection{The case when $\pi_1(M)=\Z^{*m}$}
When $\pi_1(M)$ is a free group, $M_3$ is homotopy equivalent to a wedge sum of spheres \cite{katanaga}
\[
M_3\simeq(\bigvee^m_{i=1}S^3)\vee(\bigvee^n_{j=1}S^2)\vee(\bigvee^m_{k=1}S^1).
\]
Using Theorem~\ref{gaugedecomp} we can calculate the homotopy types of $\Sigma M$ and $\G_t(M)$.

\begin{thm}\label{gauge decomp_free}
Suppose $\pi_1(M)\cong\Z^{*m}$. Let the rank of $H^2(M)$ be $n$. If $\Sigma f$ is null-homotopic, then there are homotopy equivalences
\[
\begin{array}{c}
\displaystyle{\Sigma M\simeq S^5\vee(\bigvee^m_{i=1}S^4)\vee(\bigvee^n_{j=1}S^3)\vee(\bigvee^m_{k=1}S^2)}\\[15pt]
\displaystyle{\G_t(M)\simeq\G_t(S^4)\times\prod^m_{i=1}\Omega^3G\times\prod^n_{j=1}\Omega^2G\times\prod^m_{k=1}\Omega G.}
\end{array}
\]
If $\Sigma f$ is nontrivial, then there are homotopy equivalences
\[
\begin{array}{c}
\displaystyle{\Sigma M\simeq\Sigma\C\PP^2\vee(\bigvee^m_{i=1}S^4)\vee(\bigvee^{n-1}_{j=1}S^3)\vee(\bigvee^m_{k=1}S^2)}\\[15pt]
\displaystyle{\G_t(M)\simeq\G_t(\C\PP^2)\times\prod^m_{i=1}\Omega^3G\times\prod^{n-1}_{j=1}\Omega^2G\times\prod^m_{k=1}\Omega G.}
\end{array}
\]
\end{thm}

\begin{proof}
Denote the $i^{\text{th}}$ copy of $S^3$ in $M_3$ by $S^3_i$ and the $k^{\text{th}}$ copy of $S^1$ by $S^1_k$. We show that the inclusions $\varphi^3_i:S^3_i\to M_3$ and $\varphi^1_k:S^1_k\to M_3$ satisfy the hypothesis of Lemma~\ref{2_Y1 vee Y2} for all $i$ and $k$.

Let $p^3_i:M_3\to S^3_i$ and $p^1_k:M_3\to S^1_k$ be the pinch maps. Then $\Sigma p^3_i$ and $\Sigma p^1_k$ are left homotopy inverses of $\Sigma\varphi^3_i$ and $\Sigma\varphi^1_k$. Moreover, $\Sigma p^3_i\circ\Sigma f\simeq\Sigma(p^3_i\circ f)$ is null-homotopic since~$f$ induces a trivial homomorphism $f^*:H^3(M_3)\to H^3(S^3)$, and $\Sigma p^1_k\circ\Sigma f\simeq\Sigma(p^1_k\circ f)$ is null-homotopic since $p^1_k\circ f$ is null-homotopic by $\pi_3(S^1)=0$. Apply Lemma~\ref{2_Y1 vee Y2} and get
\[
\begin{array}{c c c}
\displaystyle{\Sigma M\simeq\Sigma M'\vee(\bigvee^m_{i=1}S^4)\vee(\bigvee^m_{k=1}S^2)}
&\text{and}
&\displaystyle{\G_t(M)\simeq\G_t(M')\times\prod^m_{i=1}\Omega^3G\times\prod^m_{k=1}\Omega G.}
\end{array}
\]
where $M'$ is the cofiber of the inclusion $(\bigvee^m_{i=1}S^1)\vee(\bigvee^m_{k=1}S^3)\hookrightarrow M$. Observe that the 3-skeleton of $M'$ is homotopy equivalent to $\bigvee^n_{i=1}S^2$. Since $H^2(S^1_i)$ and $H^2(S^3_k)$ are zero, Lemma~\ref{q isom on H free} implies that $M'$ satisfies Poincar\'{e} Duality. Let $f'$ be the attaching map of the 4-cell in $M'_3$. By Proposition~\ref{2_eg_pi1=0}, if $\Sigma f'$ is null-homotopic, then we have
\[
\begin{array}{c c c}
\displaystyle{\Sigma M'\simeq S^5\vee(\bigvee^n_{i=1}S^3)}
&\text{and}
&\displaystyle{\G_t(M')\simeq\G_t(S^4)\times\prod^n_{i=1}\Omega^2G}.
\end{array}
\]
If $\Sigma f'$ is nontrivial, then we have
\[
\begin{array}{c c c}
\displaystyle{\Sigma M'\simeq\Sigma\C\PP^2\vee(\bigvee^{n-1}_{i=1}S^3)}
&\text{and}
&\displaystyle{\G_t(M')\simeq\G_t(\C\PP^2)\times\prod^{n-1}_{i=1}\Omega^2G}.
\end{array}
\]
By Lemma~\ref{lemma_f' and f} $\Sigma f'$ is null-homotopic if and only if $\Sigma f$ is null-homotopic. Combining these homotopy equivalences gives the theorem.
\end{proof}

\subsection{The case when $\pi_1(M)=\Z/p^r\Z$}
Recall that an $n$-dimensional Moore space~$P^n(k)$ is the cofiber of the degree-$k$ map \mbox{$S^{n-1}\overset{k}{\longrightarrow}S^{n-1}$} for $n\geq2$. With integral coefficients, $\tilde{H}_i(P^n(k))$ is $\Z/k\Z$ for $i=n-1$ and is zero otherwise. With mod-$k$ coefficients, $\tilde{H}_i(P^n(k);\Z/k\Z)$ is $\Z/k\Z$ for $i=n-1$ and $n$ and is zero otherwise. Let $u$ be a generator of $\tilde{H}_{n-1}(P^n(k))$. By the Universal Coefficient Theorem, $\tilde{H}_{n-1}(P^n(k);\Z/k\Z)$ is generated by the mod-$k$ reduction $\bar{u}$ of $u$, and $\tilde{H}_n(P^n(k);\Z/k\Z)$ is generated by $\bar{v}=\beta\bar{u}$, where $\beta$ is the Bockstein homomorphism.

For any space $X$, the mod-$k$ homotopy group $\pi_n(X;\Z/k\Z)$ is defined to be $[P^n(k), X]$. When $n\geq3$, $\pi_n(X;\Z/k\Z)$ has a group structure induced by the comultiplication of $P^n(k)$ and when $n\geq4$, $\pi_n(X;\Z/k\Z)$ is abelian. There are two associated homomorphisms: the mod-$k$ Hurewicz homomorphism
\[
\bar{h}:\pi_n(X;\Z/k\Z)\to H_n(X;\Z/k\Z),
\]
which is defined to be $\bar{h}(f)=f_*(\bar{v})$ for $f\in\pi_n(X;\Z/k\Z)$, and the homotopy Bockstein homomorphism
\[
\bar{\beta}_{\pi}:\pi_n(X;\Z/k\Z)\to\pi_{n-1}(X),
\]
which is defined to be $\bar{\beta}_{\pi}(f)=\imath^*\circ f$ and $\imath:S^{n-1}\to P^n(k)$ is the inclusion. They are compatible with the standard Hurewicz homomorphism $h$ and Bockstein homomorphisms $\beta$ in the commutative diagram \cite{neisendorfer1}
\begin{equation}\label{3.2_ker of Bockstein homo}
\xymatrix{
\cdots\ar[r]	&\pi_{n+1}(X;\Z/k\Z)\ar[r]^-{\beta_\pi}\ar[d]^-{\bar{h}}	&\pi_n(X)\ar[r]^-{k}\ar[d]^-{h}	&\pi_n(X)\ar[r]\ar[d]^-{h}	&\pi_n(X;\Z/k\Z)\ar[r]\ar[d]^-{\bar{h}}		&\cdots\\
\cdots\ar[r]	&H_{n+1}(X;\Z/k\Z)\ar[r]^-{\beta}							&H_n(X)\ar[r]^-{k}				&H_n(X)\ar[r]				&H_n(X;\Z/k\Z)\ar[r]				&\cdots
}
\end{equation}

For any map $g:P^3(k)\to P^3(k)$, let $C_g$ be its cofiber, and let $a$ and $b$ be generators of~$H^2(C_g;\Z/k\Z)$ and $H^4(C_g;\Z/k\Z)$. Then the mod-$k$ Hopf invariant $\bar{H}(g)\in\Z/k\Z$ is defined by the formula $a\cup a\equiv\bar{H}(g)b\pmod{k}$.

\begin{lemma}\cite[Corollary 11.12]{neisendorfer2}\label{mod p hurewicz SES}
Let $p$ be an odd prime and let $g:P^3(p^r)\to P^3(p^r)$ be a map in the kernel of $\bar{h}$. Then $g$ is null-homotopic if and only if its mod-$p^r$ Hopf invariant~$\bar{H}(g)$ is zero.
\end{lemma}

Back to the calculation of $\G_t(M)$. When $\pi_1(M)=\Z/p^r\Z$, Poincar\'{e} Duality and the Universal Coefficient Theorem imply the homology groups of $M$ are as follows:
\[
H_i(M;\Z)=\begin{cases}
\Z							&i=0,4\\
\Z/p^r\Z					&i=1\\
\Z^{\oplus n}\oplus\Z/p^r\Z	&i=2\\
0							&\text{else}.
\end{cases}
\]
Now we calculate the homotopy type of $\Sigma M$ to find a possible subcomplex $Y$ satisfying the hypothesis of Theorem~\ref{gaugedecomp}.

\begin{lemma}\label{pi_Pm}
Let $p$ be an odd prime. Then $\pi_4(P^4(p^r))$ and $\pi_4(P^3(p^r))$ are trivial.
\end{lemma}

\begin{proof}
After localizing away from $p$, any $P^n(p^r)$ is contractible so $\pi_4(P^4(p^r))$ and $\pi_4(P^3(p^r))$ are $p$-torsion. Localize at $p$ and consider the long exact sequence of homotopy groups for the pair~$(P^4(p^r), S^3)$:
\[
\cdots\longrightarrow\pi_4(S^3)\longrightarrow\pi_4(P^4(p^r))\overset{j_*}{\longrightarrow}\pi_4(P^4(p^r), S^3)\longrightarrow\cdots
\]
The pair $(P^4(p^r), S^3)$ is 3-connected, so $\pi_4(P^4(p^r), S^3)$ is $\Z$ by Hurewicz Theorem. Since~$\pi_4(S^3)$ is trivial at odd primes, $j_*$ is an injection. But~$\pi_4(P^4(p^r))$ is torsion, so this injection only makes sense if $\pi_4(P^4(p^r))$ is trivial.

Now we calculate $\pi_4(P^3(p^r))$. Let $F^3\{p^r\}$ be the homotopy fiber of the pinch map $P^3(p^r)\longrightarrow S^3$. Then $\pi_4(P^3(p^r))$ equals $\pi_4(F^3\{p^r\})$ since $\pi_4(\Omega S^3)$ and $\pi_4(S^3)$ are trivial at odd primes. By \cite[Proposition 11.7.1]{neisendorfer1}, there is a $p$-local homotopy equivalence
\[
\Omega F^3\{p^r\}\simeq S^1\times\Omega\Sigma(\bigvee_\alpha P^{n_\alpha}(p^r))\times\prod_jS^{2p^j-1}\{p^{r+1}\},
\]
where $n_\alpha$ is either~3 or greater than 4, and $S^n\{p^{r+1}\}$ is the homotopy fiber of the degree map $p^{r+1}:S^n\longrightarrow S^n$. Since $\pi_4(F^3\{p^r\})$ is $\pi_3(\Omega F^3\{p^r\})$, we need to calculate the third homotopy group of each factor on the right hand side. The first factor $\pi_3(S^1)$ and the last factor is trivial since $S^{2p^j-1}\{p^{r+1}\}$ is~3-connected for $j\geq1$. For the remaining factor of $\Omega F^3\{p^r\}$, consider the string of isomorphisms
\begin{eqnarray*}
\pi_3(\Omega\Sigma(\bigvee_\alpha P^{n_{\alpha}}(p^{r})))
&\cong&\pi_{3}(\prod_{\alpha}\Omega\Sigma P^{n_{\alpha}}(p^{r}))\\
&\cong&\pi_{3}(\Omega\Sigma P^{3}(p^{r}))\\
&\cong&\pi_{4}(P^{4}(p^{r}))\\
&\cong&0. 
\end{eqnarray*}
The first isomorphism is obtained from the Hilton-Milnor Theorem using dimension and connectivity considerations. The second isomorphism holds since only one $n_{\alpha}$ equals 3 while the rest are strictly larger than 4. The third isomorphism holds by adjunction, and the fourth isomorphism holds as we have already seen that $\pi_4(P^4(p^r))=0$. Therefore $\pi_4(F^3\{p^r\})$ is trivial and so is $\pi_4(P^3(p^r))$.
\end{proof}

\begin{lemma}\label{3_pi=Zp Sigma M_3}
If $\pi_1(M)\cong\Z/p^r\Z$, then there is a homotopy equivalence
\[
\Sigma M_3\simeq P^4(p^r)\vee(\bigvee^n_{i=1}S^3)\vee P^3(p^r).
\]
\end{lemma}

\begin{proof}
Since $\Sigma M_3$ is simply-connected, it has a minimal cell structure. By \cite[Proposition 4H.3]{hatcher} $\Sigma M_3$ is homotopy equivalent to the cofiber of a map
\[
g:P^3(p^r)\vee(\bigvee^n_{i=1}S^2)\to P^3(p^r)
\]
such that $g$ induces a trivial homomorphism $g_*:H_2(P^3(p^r)\vee(\bigvee^n_{i=1}S^2))\to H_2(P^3(p^r))$. We claim that $g$ is null-homotopic. To distinguish the 2-spheres in the wedge sum, denote the~$i^{\text{th}}$ copy by $S^2_i$. Since
\[
[P^3(p^r)\vee(\bigvee^n_{i=1}S^2_i), P^3(p^r)]\cong[P^3(p^r), P^3(p^r)]\oplus\left(\bigoplus^n_{i=1}[S^2_i, P^3(p^r)]\right),
\]
we write $g=g'\oplus(\bigoplus^n_{i=1}g''_i)$, where $g'\in[P^3(p^r), P^3(p^r)]$ and $g''_i\in[S^2_i, P^3(p^r)]$.

Consider the commutative diagram
\[\xymatrix{
\pi_2(S^2_i)\ar[r]^-{g''_i}\ar[d]_-{h}	&\pi_2(P^3(p^r))\ar[d]_-{h}\\
H_2(S^2_i)\ar[r]^-{(g''_i)_*}			&H_2(P^3(p^r)).
}\]
Both Hurewicz homomorphisms $h$ are isomorphisms by Hurewicz Theorem. Since $g_*$ is trivial, so is $(g''_i)_*$. The diagram implies that $g''_i$ is null-homotopic. Therefore $\Sigma M_3$ is homotopy equivalent to $C_{g'}\vee(\bigvee^n_{i=1}S^3)$, where $C_{g'}$ is the cofiber of $g'$. It suffices to show that $g'$ is null-homotopic.

Consider the commutative diagram
\[\xymatrix{
\pi_3(P^3(p^r);\Z/p^r\Z)\ar[r]^-{\bar{h}}\ar[d]^-{\beta_\pi}	&H_3(P^3(p^r);\Z/p^r\Z)\ar[d]^-{\beta}\\
\pi_2(P^3(p^r))\ar[r]^-{h}										&H_2(P^3(p^r))
}\]
from~(\ref{3.2_ker of Bockstein homo}). The induced homomorphism $g'_*$ is trivial and we have
\[
h\circ\beta_\pi(g')=\beta\circ\bar{h}(g')=\beta\circ((g')_*v)=0.
\]
Observe that $\beta:H_3(P^3(p^r);\Z/p^r\Z)\to H_2(P^3(p^r))$ is an isomorphism in this case, so~$g'$ is in the kernel of $\bar{h}$. Since $C_{g'}$ retracts off the suspension $\Sigma M_3$, it is a co-H-space and~$H^*(C_{g'};\Z/p^r\Z)$ has trivial cup products. Therefore the mod-$p^r$ Hopf invariant $\bar{H}(g')$ is zero and~$g'$ is null-homotopic by Lemma~\ref{mod p hurewicz SES}.
\end{proof}

Lemma~\ref{3_pi=Zp Sigma M_3} says that $\Sigma M_3$ contains $P^3(p^r)\vee P^4(p^r)$ as its wedge summands. This, however, does not necessarily imply that $M_3$ contains $P^2(p^r)\vee P^3(p^r)$ since $M$ is not simply-connected.

\begin{lemma}\label{2 Moore sp P}
If $\pi_1(M)\cong\Z/p^r\Z$, then there exists a map $\epsilon:P^2(p^r)\to M_3$ satisfying the hypothesis of Theorem~\ref{gaugedecomp} and its cofiber $C_{\epsilon}$ is simply-connected.
\end{lemma}

\begin{proof}
By the Cellular Approximation Theorem, $\pi_1(M_3)$ equals $\pi_1(M)\cong\Z/p^r\Z$. Let~\mbox{$j:S^1\to M_3$} represent a generator of $\pi_1(M_3)$. It has order $p^r$, so there exists an extension $\epsilon:P^2(p^r)\to M_3$.
\[\xymatrix{
S^1\ar[r]^-{p^r}	&S^1\ar[r]\ar[d]^-{j}	&P^2(p^r)\ar@{-->}[dl]^{\epsilon}\\
					&M_3					&
}\]
Since $\pi_1(M_3)$ is abelian, $H_1(M_3)$ is $\pi_1(M_3)\cong\Z/p^r\Z$ by Hurewicz Theorem and the induced map $\epsilon_*:H_1(P^2(p^r))\to H_1(M_3)$ is an isomorphism. Therefore the cofiber $C_{\epsilon}$ of $\epsilon$ has~$H_1(C_{\epsilon})=0$, implying that $C_{\epsilon}$ is simply-connected.

Now we show that $\Sigma\epsilon$ has a left homotopy inverse. Let $\psi$ be the composition
\[
\psi:\Sigma M_3\simeq P^4(p^r)\vee(\bigvee^n_{i=1}S^3_i)\vee P^3(p^r)\overset{pinch}{\longrightarrow}P^3(p^r).
\]
Observe that
\[
\begin{array}{c c c}
\psi_*:H_2(\Sigma M_3;\Z/p^r\Z)\to H_2(P^3(p^r);\Z/p^r\Z)
&\text{and}
&(\Sigma\epsilon)_*:H_2(P^3(p^r);\Z/p^r\Z)\to H_2(\Sigma M_3;\Z/p^r\Z)
\end{array}
\]
are isomorphism, so $(\psi\circ\Sigma\epsilon)_*:H_2(P^3(p^r);\Z/p^r\Z)\to H_2(P^3(p^r);\Z/p^r\Z)$ is an isomorphism. Then Bockstein homomorphism implies $(\psi\circ\Sigma\epsilon)_*:H_3(P^3(p^r);\Z/p^r\Z)\to H_3(P^3(p^r);\Z/p^r\Z)$ is an isomorphism. Therefore $\psi\circ\Sigma\epsilon$ is a homotopy equivalence and $\psi$ is a left homotopy inverse of $\Sigma\epsilon$.

Moreover, the composition $\psi\circ\Sigma f$ is null-homotopic since $\pi_4(P^3(p^r))$ is trivial by Lemma~\ref{pi_Pm}. Therefore $\epsilon$ satisfies the hypothesis of Theorem~\ref{gaugedecomp}.
\end{proof}

Denote the pointed mapping space $Map^*(P^n(p^r), G)$ by $\Omega^nG\{p^r\}$. This notation is justified since $Map^*(P^n(p^r), G)$ is the homotopy fiber of the power map $p^r:\Omega^nG\to\Omega^nG$. Now we calculate the homotopy type of $\G_t(M)$.

\begin{thm}\label{gauge decomp_torsion}
Suppose $\pi_1(M)\cong\Z/p^r\Z$ where $p$ is an odd prime. Let the rank of $H^2(M)$ be $n$. If $\Sigma f$ is null-homotopic, then there are homotopy equivalences
\[
\begin{array}{c}
\displaystyle{\Sigma M\simeq S^5\vee P^4(p^r)\vee(\bigvee^n_{i=1}S^3)\vee P^3(p^r)}\\[15pt]
\displaystyle{\G_t(M)\simeq\G_t(S^4)\times\Omega^3G\{p^r\}\times\prod^n_{i=1}\Omega^2G\times\Omega^2G\{p^r\}.}
\end{array}
\]
If $\Sigma f$ is nontrivial, then there are homotopy equivalences
\[
\begin{array}{c}
\displaystyle{\Sigma M\simeq\Sigma\C\PP^2\vee P^4(p^r)\vee(\bigvee^{n-1}_{i=1}S^3)\vee P^3(p^r)}\\[15pt]
\displaystyle{\G_t(M)\simeq\G_t(\C\PP^2)\times\Omega^3G\{p^r\}\times\prod^{n-1}_{i=1}\Omega^2G\times\Omega^2G\{p^r\}.}
\end{array}
\]
\end{thm}

\begin{proof}
We decompose $\G_t(M)$ in the following steps.

\emph{Step 1}: By Lemma~\ref{2 Moore sp P}, there exists a map $\epsilon:P^2(p^r)\to M_3$ satisfying the hypothesis of Theorem~\ref{gaugedecomp}. Apply the theorem and get
\[
\begin{array}{c c c}
\Sigma M\simeq\Sigma M'\vee P^3(p^r)
&\text{and}
&\G_t(M)\simeq\G_t(M')\times\Omega^2G\{p^r\}.
\end{array}
\]
where $M'$ is the cofiber of $P^2(p^r)\overset{\epsilon}{\longrightarrow}M_3\hookrightarrow M$.

\emph{Step 2}: By Lemma~\ref{2 Moore sp P} $M'$ is simply-connected, so its 3-skeleton $M'_3$ is $(\bigvee^n_{i=1})S^2\vee P^3(p^r)$. We show that the inclusion $\varphi:P^3(p^r)\hookrightarrow M'_3$ satisfies the condition of Theorem~\ref{gaugedecomp}. The pinch map $\Sigma M'_3\to P^4(p^r)$ is a left homotopy inverse of $\Sigma\varphi$. Let $f'$ be the attaching map of the 4-cell in $M'$. Then the composition
\[
S^4\overset{\Sigma f'}{\longrightarrow}\Sigma M'_3\overset{pinch}{\longrightarrow}P^4(p^r)
\]
is null-homotopic since $\pi_4(P^4(p^r))$ is trivial by Lemma~\ref{pi_Pm}. Apply Theorem~\ref{gaugedecomp} and get
\[
\begin{array}{c c c}
\Sigma M'\simeq\Sigma M''\vee P^4(p^r)
&\text{and}
&\G_t(M')\simeq\G_t(M'')\times\Omega^3G\{p^r\}.
\end{array}
\]
where $M''$ is $M'/P^3(p^r)$.

\emph{Step 3}: By Lemma~\ref{3_pi=Zp Sigma M_3} the 3-skeleton of $M''$ is homotopy equivalent to $\bigvee^n_{i=1}S^2$. Since $H^2(P^2(p^r))$ and $H^2(P^3(p^r))$ are torsion, Lemma~\ref{q isom on H free} implies that $M''$ satisfies Poincar\'{e} Duality. Let $f''$ be the attaching map of the 4-cell in $M''$. By Proposition~\ref{2_eg_pi1=0}, if $\Sigma f''$ is null-homotopic, then
\[
\begin{array}{c c c}
\displaystyle{\Sigma M''\simeq S^5\vee(\bigvee^n_{i=1}S^3)}
&\text{and}
&\displaystyle{\G_t(M'')\simeq\G_t(S^4)\times\prod^n_{i=1}\Omega^2G}
\end{array}
\]
If $\Sigma f''$ is nontrivial, then
\[
\begin{array}{c c c}
\displaystyle{\Sigma M''\simeq\Sigma\C\PP^2\vee(\bigvee^{n-1}_{i=1}S^3)}
&\text{and}
&\displaystyle{\G_t(M'')\simeq\G_t(\C\PP^2)\times\prod^{n-1}_{i=1}\Omega^2G}
\end{array}
\]

\emph{Step 4}: Combining all the homotopy equivalences from Step~1 to~3 gives the theorem.
\end{proof}

\subsection{The case when $\pi_1=(\Z^{*m})*(*^n_{j=1}\Z/p_j^{r_j}\Z)$}
Suppose that $\pi_1(M)$ is $(\Z^{*m})*(*^n_{j=1}\Z/p_j^{r_j}\Z)$, where~$p_j$ is an odd prime. The Stable Decomposition Theorem \cite[Theorem 1.3]{KLT} implies that for some number $d$ there is a diffeomorphism
\begin{equation}\label{KLT decomp}
M\#_d(S^2\times S^2)\cong N\#(\#^n_{j=1}L_j),
\end{equation}
where $N$ and $L_j$'s are orientable, smooth, connected, closed 4-manifolds with $\pi_1(N)=\Z^{*m}$ and \mbox{$\pi_1(L_j)=\Z/p_j^{r_j}\Z$}. We will calculate the suspensions and gauge groups for both sides of isomorphism~(\ref{KLT decomp}).

\begin{lemma}\label{S2 x S2}
There are homotopy equivalences
\[
\begin{array}{c c c}
\displaystyle{\Sigma(M\#_d(S^2\times S^2))\simeq(\Sigma M)\vee(\bigvee^{2d}_{s=1}S^3)}
&\text{and}
&\displaystyle{\G_t(M\#_d(S^2\times S^2))\simeq\G_t(M)\times\prod^{2d}_{s=1}\Omega^2G.}
\end{array}
\]
\end{lemma}

\begin{proof}
By induction it suffices to show the lemma for the $d=1$ case. We use Lemma~\ref{connected sum} to prove it. For $M$, the inclusion $\varphi:\{m_0\}\hookrightarrow M$ of the basepoint $m_0$ obviously satisfies the hypothesis of Theorem~\ref{gaugedecomp}. For $S^2\times S^2$, take the inclusion $\varphi':S^2\vee S^2\hookrightarrow S^2\times S^2$ to be the inclusion of the 3-skeleton. Since $\Sigma(S^2\times S^2)$ is homotopy equivalent to $S^5\vee S^3\vee S^3$, $\varphi'$ satisfies the hypothesis as well. Apply Lemma~\ref{connected sum} to get
\[
\begin{array}{c c c}
\displaystyle{\Sigma(M\#(S^2\times S^2))\simeq(\Sigma C_j)\vee(\bigvee^{2}_{s=1}S^3)}
&\text{and}
&\displaystyle{\G_t(M\#(S^2\times S^2))\simeq\G_t(C_j)\times\prod^{2}_{s=1}\Omega^2G,}
\end{array}
\]
where $C_j$ is the cofiber of the inclusion $j:\{m_0\}\vee(S^2\vee S^2)\hookrightarrow M\#(S^2\times S^2)$. Consider the cofibration diagram
\[\xymatrix{
\ast\ar[r]\ar[d]			&S^3\ar@{=}[r]\ar[d]^-{f}			&S^3\ar[d]^-{f_{M}}\\
S^2\vee S^2\ar[r]\ar@{=}[d]	&M_3\vee S^2\vee S^2\ar[r]\ar[d]	&M_3\ar[d]\\
S^2\vee S^2\ar[r]^-{j}		&M\#(S^2\times S^2)\ar[r]			&M
}\]
where $f_{M}$ is the attaching map of the 4-cell in $M$. The bottom row implies that $C_j$ is homotopy equivalent to $M$ and the asserted homotopy equivalences follow.
\end{proof}

\begin{thm}\label{gauge decomp_*}
Let $\pi_1(M)\cong(\Z^{*m})*(*^n_{j=1}\Z/p_j^{r_j}\Z)$ where $p_j$ is an odd prime and let $l$ be the rank of $H^2(M)$. Then there exists a number $d$ such that $\Sigma(M\#_d(S^2\times S^2))$ is homotopy equivalent to either
\[
\begin{array}{c c}
\displaystyle{S^5\vee(\bigvee^m_{i=1}S^4)\vee(\bigvee^n_{j=1}P^4(p_j^{r_j}))\vee(\bigvee^{l+2d}_{k=1}S^3)\vee(\bigvee^n_{j'=1}P^3(p_{j'}^{r_{j'}}))\vee(\bigvee^m_{i'=1}S^2)}	&\text{or}\\[15pt]
\displaystyle{\Sigma\C\PP^2\vee(\bigvee^m_{i=1}S^4)\vee(\bigvee^n_{j=1}P^4(p_j^{r_j}))\vee(\bigvee^{l+2d-1}_{k=1}S^3)\vee(\bigvee^n_{j'=1}P^3(p_{j'}^{r_{j'}}))\vee(\bigvee^m_{i'=1}S^2).}	&
\end{array}
\]
In the first case, we have
\[
\G_t(M)\times\prod^{2d}_{s=1}\Omega^2G\simeq\G_t(S^4)\times\prod^m_{i=1}\Omega^3G\times\prod^n_{j=1}\Omega^3G\{p_j^{r_j}\}\times\prod^{l+2d}_{k=1}\Omega^2G\times\prod^n_{j'=1}\Omega^2G\{p_{j'}^{r_{j'}}\}\times\prod^m_{i'=1}\Omega G.
\]
In the second case, we have
\[
\G_t(M)\times\prod^{2d}_{s=1}\Omega^2G\simeq\G_t(\C\PP^2)\times\prod^m_{i=1}\Omega^3G\times\prod^n_{j=1}\Omega^3G\{p_j^{r_j}\}\times\prod^{l+2d-1}_{k=1}\Omega^2G\times\prod^n_{j'=1}\Omega^2G\{p_{j'}^{r_{j'}}\}\times\prod^m_{i'=1}\Omega G.
\]
\end{thm}

\begin{proof}
Denote $N\#(\#^n_{j=1}L_j)$ by $X$. Using the stable decomposition~(\ref{KLT decomp}) and Lemma~\ref{S2 x S2}, we only need to calculate $\Sigma X$ and $\G_t(X)$. Let $N_3$ and $(L_j)_3$ be the 3-skeletons of $N$ and $L_j$. Then the 3-skeleton $X_3$ of $X$ is the wedge sum $N_3\vee(\bigvee^n_{j=1}(L_j)_3)$.

\emph{Step 1}: The 3-skeleton $N_3$ is homotopy equivalent to $(\bigvee^m_{i=1}S^3)\vee(\bigvee^{l'}_{j'=1}S^2)\vee(\bigvee^m_{k=1}S^1)$. Denote the $i^{\text{th}}$ copy of $S^3$ in $N_3$ by $S^3_i$ and the $k^{\text{th}}$ copy of $S^1$ by $S^1_k$. In the proofs of Theorem~\ref{gauge decomp_free}, we show that inclusions $\varphi^3_i:S^3_i\to N_3$ and $\varphi^1_k:S^1_k\to N_3$ satisfy the hypothesis of Theorem~\ref{gaugedecomp}. For each $j$, by Lemma~\ref{2 Moore sp P} there exists a map $\epsilon_j:P^2(p_j^{r_j})\to(L_j)_3$ satisfying the hypothesis as well. Apply Lemma~\ref{connected sum} and get
\[
\begin{array}{c}
\displaystyle{\Sigma X\simeq\Sigma X'\vee\left(\bigvee^m_{i=1}S^4\right)\vee\left(\bigvee^m_{k=1}S^2\right)\vee\left(\bigvee^n_{j=1}P^3(p_j^{r_j})\right)}\\
\displaystyle{\G_t(X)\simeq\G_t(X')\times\prod^m_{i=1}\Omega^3G\times\prod^m_{k=1}\Omega G\times\prod^n_{j=1}\Omega G\{p_j^{r_j}\}}
\end{array}
\]
where $X'$ is the cofiber of the inclusion $\left(\bigvee^m_{i=1}S^3\right)\vee\left(\bigvee^m_{k=1}S^1\right)\vee\left(\bigvee^n_{j=1}P^2(p_j^{r_j})\right)\hookrightarrow X$.

\emph{Step 2}: Since $X'$ is simply-connected, its 3-skeleton $X'_3$ has a minimal cell structure
\[
X_3\simeq\left(\bigvee^{l''}_{j'=1}S^2\right)\vee\left(\bigvee^n_{j=1}P^3(p_j^{r_j})\right).
\]
We show that inclusions $\varphi_j:P^3(p_j^{r_j})\to X'_3$ satisfy the hypothesis of Theorem~\ref{gaugedecomp}. For each~$j$, the pinch maps $\Sigma X'_3\to P^4(p_j^{r_j})$ is a left homotopy inverse of $\Sigma\varphi_j$. Let $f'$ be the attaching map of the 4-cell in $X'$. Then the composition
\[
S^4\overset{\Sigma f'}{\longrightarrow}\Sigma X'_3\overset{pinch}{\longrightarrow}P^4(p_j^{r_j})
\]
is null-homotopic since $\pi_4(P^3(p_j^{r_j}))$ is trivial by Lemma~\ref{pi_Pm}. Apply Lemma~\ref{2_Y1 vee Y2} and get
\[
\begin{array}{c c c}
\displaystyle{\Sigma X'\simeq\Sigma X''\vee\left(\bigvee^n_{j=1}P^3(p_j^{r_j})\right)}
&\text{and}
&\displaystyle{\G_t(X')\simeq\G_t(X'')\times\prod^n_{j=1}\Omega^2G\{p_j^{r_j}\},}
\end{array}
\]
where $X''$ is $X'/(\bigvee^n_{j=1}P^3(p_j^{r_j}))$.

\emph{Step 3}: The 3-skeleton of $X''$ is homotopy equivalent to $\bigvee^{l''}_{j'=1}S^2$. Since the subcomplexes $S^1_i, S^3_k, P^2(p_j^{r_j})$ and $P^3(p_j^{r_j})$ in Step 1 and 2 have either zero or torsion second cohomology groups, Lemma~\ref{q isom on H free} implies that $X''$ satisfies Poincar\'{e} Duality. By Proposition~\ref{2_eg_pi1=0}, we have
\[
\begin{array}{c c c}
\displaystyle{\Sigma X''\simeq S^5\vee\left(\bigvee^{l''}_{j'=1}S^3\right)}
&\text{and}
&\displaystyle{\G_t(X'')\simeq\G_t(S^4)\times\prod^{l''}_{j'=1}\Omega^2G}
\end{array}
\]
or 
\[
\begin{array}{c c c}
\displaystyle{\Sigma X''\simeq\Sigma\C\PP^2\vee\left(\bigvee^{l''-1}_{j'=1}S^3\right)}
&\text{and}
&\displaystyle{\G_t(X'')\simeq\G_t(\C\PP^2)\times\prod^{l''-1}_{j'=1}\Omega^2G.}
\end{array}
\]

\emph{Step 4}: Combining all homotopy equivalences from Step 1 to 3, the stable decomposition~(\ref{KLT decomp}) and Lemma~\ref{S2 x S2} together imply the theorem. Furthermore, $H_2(M\#_d(S^2\times S^2))$ has rank $l=2d$, so $l''=l+2d$.
\end{proof}

\begin{remark}
The proofs of Theorem~\ref{gauge decomp_free} and \ref{gauge decomp_torsion} are also valid for any orientable 4-dimensional CW-complex with one 4-cell, while Theorem~\ref{gauge decomp_*} requires the smoothness of $M$ for the stable diffeomorphism splitting in \cite{KLT}. One can ask under what condition the stabilizing factor~$\#_d(S^2\times S^2)$ can be cancelled so that the factors $\prod^{2d}_{s=1}\Omega^2G$ can be removed from the equation. If this can be achieved, Theorem~\ref{gauge decomp_free} and \ref{gauge decomp_torsion} will be corollaries of Theorem~\ref{gauge decomp_*} when $M$ is smooth.
\end{remark}

\end{document}